\documentclass{my-lmns}

\usepackage[pdftex]{graphicx}  
\usepackage{cite}
\usepackage{url} 
\usepackage{booktabs}
\usepackage{amsmath}
\usepackage{amssymb}
\usepackage{algorithm}
\usepackage[noend]{algpseudocode}
\usepackage[singlelinecheck=false]{caption}

\makeatletter
\def\BState{\State\hskip-\ALG@thistlm}
\makeatother

\usepackage{fixltx2e}
\usepackage{xcolor}

\begin{document}

\lmnsVolume{43}{1}{2023}{1}{}

\lmnsTitle[Boundary Regional Controllability of Semilinear Systems]{%
Boundary Regional Controllability of Semilinear Systems Involving Caputo Time Fractional Derivatives}

\lmnsAuthors[A. Tajani, F.-Z. El Alaoui and D. F. M. Torres]{%
Asmae Tajani, Fatima-Zahrae El Alaoui and Delfim F. M. Torres}

\lmnsAbstract{We study boundary regional controllability problems 
for a class of semilinear fractional systems. Sufficient conditions 
for regional boundary controllability are proved by assuming that 
the associated linear system is approximately regionally boundary controllable. 
The main result is obtained by using fractional powers of an operator and 
the fixed point technique under the approximate controllability of  
the corresponding linear system in a suitable subregion of the space domain. 
An algorithm is also proposed and some numerical simulations performed 
to illustrate the effectiveness of the obtained theoretical results.}

\lmnsKeywords{Fractional derivatives and integrals, 
Nonlinear systems, 
Regional controllability, 
Semigroup operators, 
Fixed-point theorems.}

\lmnsMSC{26A33, 93B05, 93C10}


\lmnsContact{Asmae Tajani}{
Center for Research and Development in Mathematics and Applications (CIDMA), 
Department of Mathematics, University of Aveiro, 
3810-193 Aveiro, Portugal\\
and\\
TSI Team, Department of Mathematics, Faculty of Sciences,\\ 
Moulay Ismail University, 11201 Meknes, Morocco}{tajaniasmae@ua.pt}

\lmnsContact{Fatima-Zahrae El Alaoui}{
TSI Team, Department of Mathematics, Faculty of Sciences,\\ 
Moulay Ismail University, 11201 Meknes, Morocco}{f.elalaoui@umi.ac.ma}

\lmnsContact{Delfim F. M. Torres}{
Center for Research and Development in Mathematics and Applications (CIDMA), 
Department of Mathematics, University of Aveiro, 3810-193 Aveiro, Portugal}{delfim@ua.pt}

\lmnsMaketitle


\begin{center}
To Professor Andrea Bacciotti (in memoriam)
\end{center}


\section{Introduction}
\label{sec:introduction}

Lately, fractional diffusion equations have garnered increasing 
attention and established themselves as invaluable tools for 
modeling various phenomena, particularly in the realms of physics, 
chemistry, engineering, medicine, and biology \cite{11,miller,app}. 
One of the most successful applications of fractional calculus manifest 
in diffusion models that elucidate the behavior of a diffusing particle, 
exhibiting mean square displacement rates either slower or faster than 
those observed in typical diffusion processes \cite{MR3904404,MyID:386,MyID:441}. 
Indeed, anomalous diffusion is prevalent in numerous experiments, 
highlighting the superior capability 
of fractional equations to characterize intricate phenomena 
\cite{anamalous,diff}. Fractional systems can also be found in the field 
of electrochemistry, where they are employed to determine the concentration 
of analyzed electroactive species near the electrode surface with greater precision. 
This task is accomplished through the application of a fractional diffusion model 
\cite{sbati,hilfer}. Owing to its extensive range of applications, the representation 
formula for the mild solutions of fractional sub-diffusion equations has undergone 
a comprehensive study: see, e.g., \cite{zhouex,borai,theor,solution,nonlinear}
and references therein.

It is widely recognized that the mathematical field of control theory 
encompasses various concepts, one of which is the notion of controllability
\cite{MR3337805,MR3684684,MR2784697}. The concept of controllability 
was initially introduced by Kalman in 1960 
and involves guiding a system toward a desired state through the use 
of control techniques \cite{kalman}. The issue of controllability, 
for both linear and nonlinear systems, for dynamic equations 
in both finite and infinite-dimensional spaces, have been extensively investigated 
\cite{cinte,cinter2,cinter3,cont1,cont2,zhoucon}. Nonetheless, in a significant 
number of practical applications, our primary concern lies in scenarios where 
the desired state of the problem at hand is confined to an internal or boundary 
subregion within the overall spatial domain. In such cases, the concept of regional 
controllability becomes crucial \cite{kamal,55,bout1,bout2}.

An increasing number of researchers have directed their attention toward the 
study of regional controllability in linear time-fractional systems 
\cite{regbund,grcont,regana}. For nonlinear fractional systems within 
infinite-dimensional spaces, pioneering results have also been achieved 
\cite{nonl1,nonl2,major,cont,me}.

The main objective of this work is to study the regional boundary 
controllability of the following class of distributed abstract fractional 
semilinear control systems involving a Caputo fractional 
order $\alpha\in(0,1]$: 
\begin{equation}
\label{semili}
\left\{\begin{array}{lll}
{^C}D_{0^+}^\alpha y(x,t) + Ay(x,t) 
= \mathcal{F} y(x,t)+ \mathbf{B}u(t)  
&\text{in} & \Omega\times]0,T], \\ \\
\dfrac{\partial y(\xi,t)}{\partial \nu_A}=0 
& \mbox{on} & \partial\Omega\times ]0,T] ,\\ \\
y(x,0)=y_{0}(x)   \hskip2cm & \text{in}  & \Omega,
\end{array} \right.
\end{equation}
with the space domain $\Omega$  being a subset of $\mathbb{R}^n$, 
$n\leq 3$. Here ${^C}D^\alpha_{0^+}$ is the Caputo fractional derivative 
of order $\alpha$; $-A$ is the infinitesimal generator of an analytic 
semigroup $(\mathcal{T}(t))_{t\geq 0}$ on the Hilbert space 
$X:=H^1(\Omega)$; $\mathcal{F}$ is a nonlinear operator; 
the control function $u$ takes values in $U=L^p([0,T],\mathcal{U})$, 
where $p\geq 2$ and $\mathcal{U}$ is a Banach space; and $\mathbf{B}$ 
is a linear control operator from $U$ into $L^p([0,T],X)$.

Semilinear systems, described by fractional equations, 
present an array of challenges and opportunities in the context 
of regional controllability. Our work is dedicated to establishing 
sufficient conditions for regional boundary controllability of the 
Caputo system \eqref{semili}, particularly when the operator $-A$ 
generates an analytic semigroup. This endeavor leverages the tools 
of fractional calculus and the Picard fixed-point theorem.

The paper is arranged as follows. Here, we have motivated 
and introduced problem \eqref{semili} under investigation. In 
Section~\ref{sec2}, we present some preliminaries that 
will be useful throughout the manuscript. Section~\ref{sec3} 
is devoted to study the boundary controllability on a subregion 
of the boundary of the evolution domain of the system: 
under some assumptions and by using the analytic approach,
we prove results on the regional controllability 
by using the relation between internal and boundary controllability concepts 
(Theorems~\ref{my:thm1} and \ref{my:thm2}).
We proceed with Section~\ref{sec4}, presenting some numerical simulations 
that illustrate the effectiveness of the proposed methods.
Finally, we end with Section~\ref{sec5} of conclusion, 
pointing also some directions of possible future research.


\section{Preliminaries}
\label{sec2}

In this section, we recall basic definitions of fractional 
operators, the notions of regional controllability, 
as well as some lemmas that will be useful afterwards.

\begin{definition}[See \cite{11}]
The Riemann--Liouville fractional integral of order $\alpha\in(0,1]$
with the lower limit $0$, of a function $g$ at point $t>0$, is defined by
\begin{equation*}
\mathcal{I}_{0^+}^\alpha g(t)
= \dfrac{1}{\Gamma(\alpha)}\displaystyle\int_{0}^{t}(t-\zeta)^{\alpha-1}
g(\zeta)d\zeta, \quad \forall  t>0,
\end{equation*}
where $\Gamma$ is the gamma function.
\end{definition}

\begin{definition}[See \cite{11}]
The Caputo fractional derivative of order  $\alpha\in(0,1]$
with the lower limit $0$, of a function $g$ at point $t>0$, is given by
\begin{equation}
\label{cafr}
^{C}\mathrm{D}_{0^+}^{\alpha}g(t)
= \dfrac{1}{\Gamma(1-\alpha)}\displaystyle\int_{0}^{t}(t-\zeta)^{-\alpha}
\displaystyle\frac{d}{d\zeta}(g(\zeta))ds=\mathcal{I}_{0^+}^\alpha g'(t). 
\end{equation}
\end{definition}

Throughout the paper, we denote by $\rho(A)$ the resolvent set of $A$, 
assuming that $0 \in \rho(A)$. Then, for $0<q< 1$, the fractional 
power of $A$ of order $q$ is well defined, being linear and closed on its domain 
$X_q=D(A^q)$. For details, see \cite{pazy}.

Next, we present the definition of mild solution of system \eqref{semili}.

\begin{definition}[See \cite{borai}] 
Let $t\in[0,T]$ and $u\in U$. The mild solution of \eqref{semili}, 
denoted by $y_u(\cdot)$, is a continuous function from $[0,T]$ to $X$   
defined by the following expression: 
\begin{equation}
\label{sys1.sol}
\begin{split}
y_u(t) &= \mathcal{H}_\alpha(t)y_0 
+ \displaystyle\int_{0}^{t}(t-\zeta)^{\alpha-1}
\mathcal{K}_\alpha(t-\zeta)\mathcal{F} y(\zeta,u)d\zeta\\
&\quad + \displaystyle\int_{0}^{t}(t-\zeta)^{\alpha-1}
\mathcal{K}_\alpha(t-\zeta)\mathbf{B}u(\zeta)d\zeta,
\end{split} 
\end{equation}
where 
$$
\mathcal{H}_\alpha(t) = \displaystyle\int_{0}^{\infty}
\mathrm{\varphi}_\alpha(\theta)\mathcal{T}(t^\alpha\theta)d\theta, 
$$
$$
\mathcal{K}_\alpha(t) = \alpha\displaystyle\int_{0}^{\infty}
\theta\mathrm{\varphi}_\alpha(\theta)\mathcal{T}(t^\alpha\theta)d\theta, 
$$
and $ \mathrm{\varphi}_\alpha $ is a probability density function.
\end{definition}

The above operators satisfy some well-known properties.
The following lemma is useful for our purposes.

\begin{lemma}[See \cite{majorit,mea}]
\label{lemma2.4}   
For any $\alpha,q\in ]0,1]$, the following properties hold:
\begin{itemize}
\item [(i)]  The mapping $||\cdot||_{X^q}=||A^q(\cdot)||_X$ 
defines a norm in $X^q$; $(X^q,||\cdot||_{X^q} )$ is a Banach space;
and $\overline{X^q}=X$;

\item [(ii)] For every $t>0$, there exists $C_q,M_{\alpha q}>0$ such that
$$
||\mathcal{H}_\alpha(t)||_{\mathcal{L}(X,X^q)}
\leq C_{q}t^{-\alpha q},
$$
$$
||\mathcal{K}_\alpha(t)||_{\mathcal{L}(X,X^q)}
\leq \dfrac{M_{\alpha q}}{t^{\alpha q}};
$$

\item[(iii)] If $\mathcal{E}_\alpha(t)=t^{\alpha-1}\mathcal{K}_\alpha(t)$, then 
$$
\mathcal{E}_\alpha\in L^1([0,T],\mathcal{L}(X,X^q)).
$$
\end{itemize}
\end{lemma} 

Throughout the text,
$A_1:=||\mathcal{E}_\alpha||_{L^1([0,T],\mathcal{L}(X,X^q))}$
with $\mathcal{E}_\alpha(t)=t^{\alpha-1}\mathcal{K}_\alpha(t)$.

Let $\omega$ be a subset of $ \Omega$. The restriction operator 
in $\omega$ is defined as follows:
$$
\begin{array}{llll}
\chi_{_{\omega}}:& H^1(\Omega) &\longrightarrow & H^1(\omega)\\
&y & \longmapsto &  y_{|_{\omega}}.		
\end{array}$$

We now introduce the notion of regional controllability.

\begin{definition}
For any element $d_s$ of $H^1(\omega)$, if there exists a control 
$u\in U$ such that $\chi_\omega y_u(T)=d_s$, then we say that
system \eqref{semili} is exactly $\omega$-controllable (i.e.,  
regionally exactly controllable in $\omega$).
On the other hand, if for all $d_s$ in $H^1(\omega)$ and 
for all $\epsilon>0$ there exists a control $u\in U$ such that 
$||\chi{_{\omega}} y_u(T)-d_s||_{H^1(\omega)}\leq \epsilon$, 
then we say that system \eqref{semili} is approximately $\omega$-controllable 
(i.e., regionally approximately controllable in $\omega$).
\end{definition}

Consider the trace operator $\gamma_0$ from $H^1(\Omega)$  
to $H^{\frac{1}{2}}(\partial\Omega)$, which is a continuous 
linear onto operator. Let us denote by $\gamma_0^*$  it's adjoint. 
For $\mathrm{\varGamma}$ a subset of $\partial\Omega$, 
we define the restriction operator $\chi_{_{\mathrm{\varGamma}}}$ by
$$
\begin{array}{rlll}
\chi_{_{\mathrm{\varGamma}}} :& H^{\frac{1}{2}}(\partial\Omega) 
&\longrightarrow & H^{\frac{1}{2}}(\mathrm{\varGamma})\\
&y & \longmapsto &  y_{|_{\mathrm{\varGamma}}},		
\end{array}
$$
and we denote by $\chi^*_{_{\mathrm{\varGamma}}}$ it's adjoint.

Next, we give the definition of regional boundary 
controllability for system \eqref{semili}.

\begin{definition}
We say that system \eqref{semili} is exactly (resp. approximately)  
boundary regionally controllable on $\mathrm{\Gamma}$ 
($\mathcal{B}$-controllable on $\mathrm{\Gamma}$) if 
$$
\forall z_d\in H^{\frac{1}{2}}(\mathrm{\varGamma}) 
\ \exists u\in U \:\: \mbox{ such that } \: \quad 
\chi_{_{\mathrm{\varGamma}}}(\gamma_0 y_u(T))=z_d
$$ 
$$ 
\left(\mbox{resp.} \qquad \forall z_d
\in H^{\frac{1}{2}}(\mathrm{\varGamma}) 
\ \exists \varepsilon>0 \  \exists u\in U \ \  
||\chi_{_{\mathrm{\varGamma}}}(\gamma_0 y_u(T))
-z_d||_{H^{\frac{1}{2}}(\mathrm{\varGamma})} 
\leq \varepsilon\right). 
$$
\end{definition}


\section{Regional $\mathcal{B}$-Controllability}
\label{sec3}

In this section, we study the possibility of finding a control 
function that steers system \eqref{semili} to a desired state
$z_d$ on  $\mathrm{\varGamma}$.
  
Let us define the operator $H_{\mathrm{\varGamma}}^\alpha$ 
from $U$ into $H^{\frac{1}{2}}(\mathrm{\varGamma})$  by
$$
H_{\mathrm{\varGamma}}^\alpha u=\chi_{_{\mathrm{\varGamma}}}
\gamma_0 (\mathcal{E}_\alpha\ast \mathbf{B}u), \quad \forall u\in U,
$$
where $\ast$ is the convolution operation.

In the next result, we suppose that the associate linear system 
to \eqref{semili} (i.e., system \eqref{semili} with $\mathcal{F} \equiv 0$) 
is approximately $\mathcal{B}$-controllable on $\mathrm{\varGamma}$.
Theorem~\ref{my:thm1} presents a direct result concerning the 
boundary controllability of system \eqref{semili} on $\mathrm{\varGamma}$.

\begin{theorem}  
\label{my:thm1}
If the subset  $\mbox{Im}\, H_{ \mathrm{\varGamma}}^\alpha$ 
is closed and contains the element 
$$
z_d- \chi_{_{\mathrm{\varGamma}}}\gamma_0 \mathcal{H}_\alpha(T)y_0
- \chi_{_{\mathrm{\varGamma}}}\gamma_0 (\mathcal{E}_\alpha\ast 
\mathcal{F} y_{\tilde{u}})(T),
$$  
then the exact $\mathcal{B}$-controllability on $\mathrm{\Gamma}$ 
of system \eqref{semili} at time $T$ is obtained 
by means of the control function
$$
\tilde{u}(t)=\mathrm{H}_{\mathrm{\varGamma}}^{\alpha^\dag}[z_d
- \chi_{_{\mathrm{\varGamma}}}\gamma_0 \mathcal{H}_\alpha(t)
y_0-\chi_{_{\mathrm{\varGamma}}}\gamma_0 (\mathcal{E}_\alpha\ast 
\mathcal{F} y_{\tilde{u}})(t)],
$$
where $H_{\mathrm{\varGamma}}^{\alpha^\dag}$ represents the pseudo 
inverse operator defined by 
$H_{\mathrm{\varGamma}}^{\alpha^\dag}:=
H_{\mathrm{\varGamma}}^{\alpha^*} \left(H_{\mathrm{\varGamma}}^\alpha 
H_{\mathrm{\varGamma}}^{\alpha^*}\right)^{-1}$.
\end{theorem}

\begin{proof} 
The mild solution of \eqref{semili} associated with the control
$\tilde{u}$  is given by
\begin{equation}
\label{solreduit}
y_{\tilde{u}}(t) = \mathcal{H}_\alpha(t)y_0 + (\mathcal{E}_\alpha\ast 
\mathcal{F} y_{\tilde{u}})(t) + (\mathcal{E}_\alpha\ast \mathbf{B}\tilde{u})(t),
\end{equation}
which implies that 
$$
\chi_{_{\mathrm{\varGamma}}}\gamma_0 y_{\tilde{u}}(T)
=\chi_{_{\mathrm{\varGamma}}}\gamma_0( \mathcal{H}_\alpha(T)y_0 
+(\mathcal{E}_\alpha\ast \mathcal{F} y_{\tilde{u}})(T)) 
+ H_{\mathrm{\varGamma}}^{\alpha}\tilde{u}(T).
$$
Therefore,
\begin{equation*}
\begin{array}{lll}
\chi_{_{\mathrm{\varGamma}}}\gamma_0 y_{\tilde{u}}(T) 
= \chi_{_{\mathrm{\varGamma}}}\gamma_0( \mathcal{H}_\alpha(T)y_0 
+(\mathcal{E}_\alpha\ast \mathcal{F} y_{\tilde{u}})(T))\\
\qquad + H_{\mathrm{\varGamma}}^{\alpha}
\mathrm{H}_{\mathrm{\varGamma}}^{\alpha^\dag}[z_d
- \chi_{_{\mathrm{\varGamma}}}\gamma_0 (\mathcal{H}_\alpha(T)y_0
-(\mathcal{E}_\alpha\ast \mathcal{F} y_{\tilde{u}})(T))].
\end{array} 
\end{equation*}
Since $H_{\mathrm{\varGamma}}^{\alpha}\mathrm{H}_{\mathrm{\varGamma}}^{\alpha^\dag}$ 
is the orthogonal projection on $\mbox{Im}H_{\mathrm{\varGamma}}^{\alpha}$ and 
$$
[z_d- \chi_{_{\mathrm{\varGamma}}}\gamma_0 (\mathcal{H}_\alpha(T)y_0
- (\mathcal{E}_\alpha\ast \mathcal{F} y_{\tilde{u}})(T))]
\in \mbox{Im} H_{\mathrm{\varGamma}}^\alpha,
$$
then $\chi_{_{\mathrm{\varGamma}}}\gamma_0 y_{\tilde{u}}(T) =z_d$. 
The proof is complete.
\end{proof}

Next, we prove the regional $\mathcal{B}$-controllability  
in the analytical setting by using the internal regional 
controllability result. More precisely, we formulate and 
establish some conditions for the regional $\mathcal{B}$-controllability 
on $\mathrm{\varGamma}$ of system \eqref{semili} with $y_0=0$. 
To do this, we first prove the connection between internal and boundary 
regional controllability, where we choose  a suitable sub-region $\omega_c$. 
Secondly, we show that, under some  assumptions, the internal regional 
controllability in $\omega_c$ is implied by approximate 
regional controllability of the corresponding linear system.

Let $\omega_c$  be a subregion of $\Omega$ such that 
$\mathrm{\varGamma}\subseteq \partial\omega_c$.
Proposition~\ref{prop} presents a link between 
internal regional controllability and the boundary one.

\begin{proposition}  
\label{prop}
The exact (resp. approximate) $\omega_c$-controllability 
of system \eqref{semili} implies the exact (resp. approximate) 
regional $\mathcal{B}$-controllability on $\mathrm{\varGamma}$.
\end{proposition}

\begin{proof}
We define the operator $R$ from 
$H^{\frac{1}{2}}(\partial\omega_c)$ into $ H^1(\omega_c)$ 
such that $\gamma_{0_{|\omega_c}} Rg=g$  
for all $g\in H^{\frac{1}{2}}(\partial\omega_c)$. 
Let $z_d\in H^{\frac{1}{2}}(\mathrm{\varGamma})$. 
By using the trace theorem, there exists $R\overline{z}_d\in H^1(\omega_c)$ 
with a bounded support such that 
$\gamma_{0_{|\omega_c}} R\overline{z}_d=\overline{z}_d$. Then, \par
(i) If system \eqref{semili} is exactly $\omega_c-$controllable, 
then, for any $d_s\in H^1(\omega_c)$, there exists $u\in U$  
such that $\chi_{\omega_c} y_u(T)=d_s$. Since 
$R\overline{z}_d\in H^1(\omega_c)$, we get that
$$
\exists u\in U \quad \mbox{such that} 
\quad \gamma_{0_{|\omega_c}}(\chi_\omega y_u(T))
= \overline{z}_d.  
$$ 
Thus, $\chi_{_{{\mathrm{\varGamma}_{|\partial\omega_c}}}}
\gamma_{0_{|\omega_c}}(\chi_{\omega_c} y_u(T))= z_d$ and 
$\chi_{{\mathrm{\varGamma}}}\gamma_{0}( y_u(T))= z_d$.\par  
(ii) In the case where system \eqref{semili} is approximately regionally  
controllable in $\omega_c$, for any $d_s\in H^1(\omega_c)$ and for all 
$\varepsilon>0$, there exists a control $u\in U$ such that 
$||\chi{_{\omega_c}} y_u(T)-d_s||_{H^1(\omega_c)}\leq \varepsilon$.
We have $R\overline{z}_d\in H^1(\omega_c)$ and
$$
\forall \varepsilon>0 \quad \exists u\in U 
\quad |||\chi{_{\omega_c}} y_u(T)-R\overline{z}_d||_{H^1(\omega_c)}
\leq \varepsilon.
$$ 
Moreover, by the continuity of  $\gamma_0$ on $H^1(\omega_c)$, we have 
$$
||\gamma_{0_{|\omega_c}}\chi{_{\omega_c}} y_u(T)
-\overline{z}_d||_{H^{\frac{1}{2}}(\partial\omega_c)}\leq \varepsilon.
$$
Therefore, 
$$
||\chi_{_{{\mathrm{\varGamma}_{|\omega_c}}}}\gamma_{0_{|\omega_c}}
\chi{_{\omega_c}} y_u(T)-z_d||_{H^{\frac{1}{2}}(\mathrm{\varGamma})}
\leq \varepsilon,
$$
and system \eqref{semili} is approximate boundary regional controllable 
on $\mathrm{\varGamma}$.
\end{proof}

To formulate and prove a result of regional controllability 
in a subregion $\omega_c$ such that 
$\mathrm{\varGamma}\subseteq \partial\omega_c$,
we make use of the following hypotheses:\\
$(H_1)$ The corresponding linear system of \eqref{semili} 
is approximately $\omega_c-$controllable.\\
$(H_2)$ The nonlinear operator 
$\mathcal{F} :L^p([0,T],X^q)\longrightarrow L^p([0,T],X)$ 
satisfies
\begin{equation*}
\left\{\begin{array}{lll}
\mathcal{F}(0)=0,& &\\
||\mathcal{F} z-\mathcal{F} y||_{L^p([0,T],X)} 
& \leq &F_N(||z||,||y||)||z-y||_{L^p([0,T],X^q)},
\end{array} \right.
\end{equation*}
where $F_N:\mathbb{R}^+ \times \mathbb{R}^+  
\longrightarrow \mathbb{R}^+ $ is such that
$$
\displaystyle\lim\limits_{(\sigma_1,\sigma_2)
\rightarrow(0,0)}F_N(\sigma_1,\sigma_2)=0.
$$

We introduce the following operators:
$$
\begin{array}{llll}
\mathrm{H}_{\omega_c}^\alpha:
& U &\longrightarrow & H^1(\omega_c)\\ 
&u & \longmapsto &  	\chi_{\omega_c} 
(\mathcal{E}_\alpha\ast \mathbf{B} u),	
\end{array}
$$
and
$$
\Psi(d_s,u)=\mathrm{H}^{\alpha^\dag}_{\omega_c}(d_s-\chi_\omega 
(\mathcal{E}_\alpha\ast \mathcal{F} y_u)(T)) 
\quad d_s\in \mbox{Im} (H_{\omega_c}^\alpha), u\in U,
$$
where $H_{\omega_c}^{\alpha^\dag}$ 
is the pseudo-inverse operator of $ H_{\omega_c}^\alpha$.

The following lemmas are also needed.

\begin{lemma}[See \cite{me}]
\label{lemma 3.3} 
For any $d_s\in \mbox{Im} (H_{\omega_c}^\alpha)$, 
the mapping 
$$
||d_s||_{\mbox{Im} (H_{\omega_c}^\alpha)}
=||H_{\omega_c}^{\alpha^\dag} d_s||_{_U}
$$ 
is a norm on $\mbox{Im} (H_{\omega_c}^\alpha)$.
\end{lemma}

\begin{lemma}
\label{lemma3.4}  
If the control operator is bounded, then the control 
function $u$ satisfies the inequality
\begin{equation}
\label{condu}
||\mathcal{E}_\alpha\ast \mathbf{B}u||_{L^p([0,T],X^q)}
\leq \mu ||u||_U, \qquad \mu>0.
\end{equation}
\end{lemma}

\begin{proof}  
The result holds with 
$\mu=A_1 ||\mathbf{B}||_{\mathcal{L}(X,\mathcal{U})}$ 
by Young's inequality \cite{Vector-valued}.
\end{proof}

\begin{remark} 
If $\mathbf{B}$ is unbounded, then we suppose that the  
inequality \eqref{condu} holds.
\end{remark}

\begin{lemma}
\label{lemma3.6} 
Assume that hypothesis ($H_2$) holds. Then there exists 
$\kappa>0$ and $m_\kappa>0$ such that 
$f: u\longrightarrow y_u$ is a Lipschitz mapping from 
$\mathrm{B}(0,m_\kappa)$ to $\mathrm{B}(0,\kappa)$,
where $\mathrm{B}(0,r)$ is the ball of center $0$ 
and radius $r>0$.
\end{lemma}

\begin{proof}  
First, we show that there exists $\kappa>0$ such that 
\begin{equation} 
\label{cste1}
m_\kappa=\dfrac{\kappa}{\mu}(1-A_1\: 
\underset{\theta\leq \kappa}{\sup} 
\: F_N(\theta,0))>0.
\end{equation}
We have, by hypothesis ($H_2$), that
$$
\displaystyle\lim\limits_{(\sigma_1,\sigma_2)
\rightarrow(0,0)}F_N(\sigma_1,\sigma_2)=0.
$$
Then $\exists \:  \kappa>0$ 
and $\exists \: \nu >0$ such that
\begin{equation}
\label{majoration}
F_N(\sigma_1,\sigma_2) < \nu< \dfrac{1}{A_2+A_1 } 
\qquad \forall \sigma_1,\sigma_2\leq \kappa,
\end{equation}
where $A_2>0$. This gives
$$
\underset{\sigma_i\leq \kappa}{\sup} 
\: F_N(\sigma_1,\sigma_2) \leq \nu< \dfrac{1}{A_2+A_1}
$$
with $A_1\underset{\sigma_i\leq \kappa}{\sup} \:     
F_N(\sigma_1,\sigma_2)<1$. In particular, $m_\kappa>0$.
Let	$u,v \in \mathrm{B}(0,m_\kappa)$. Then, 
\begin{equation*}
|| y_u-y_v||_{_{L^p([0,T],X^q)}}
\leq ||\mathcal{E}_\alpha\ast (\mathcal{F} y_u
-\mathcal{F} y_v)||_{_{L^p([0,T],X^q)}}
+||\mathcal{E}_\alpha\ast \mathbf{B}(u-v)||_{_{L^p([0,T],X^q)}}.
\end{equation*}
Using hypothesis ($H_2$) and Lemma~\ref{lemma 3.3}, we obtain that
\begin{equation}
\begin{split}
||y_u-y_v||_{_{L^p([0,T],X^q)}}
&\leq A_1 \underset{\sigma_i
\leq \kappa}{\sup} \: F_N(\sigma_1,\sigma_2)\\
&\quad \times  || y_u-y_v||_{_{L^p([0,T],X^q)}}
+\mu||u-v||_U.
\end{split}
\end{equation}
Then, $f$ is Lipschitz with the Lipschitz constant
$\dfrac{\mu}{1-A_1\underset{\sigma_i\leq \kappa}{\sup} 
\: F_N(\sigma_1,\sigma_2)}$.
\end{proof}

We prove that system \eqref{semili} is $\omega_c$-controllable 
by studying the existence of a fixed point for operator $\Psi(d_s,\cdot)$.

\begin{theorem}  
\label{my:thm2}	
If hypotheses ($H_1$) and ($H_2$) hold, together with
\begin{equation}
||\chi_{\omega_c} \mathcal{E}_\alpha(\cdot)||_{\mathcal{L}(X,
\mbox{Im} (H_{\omega_c}^{^\alpha}))}=g_\alpha\in L^s([0,T]), 
\quad \dfrac{1}{p}+\dfrac{1}{s}=1,
\end{equation}
then there exists $\kappa>0$, $m_\kappa>0$, and $\rho_\kappa>0$  
such that, for any element $d_s$ of $\mathrm{B}(0,\rho_\kappa)$, 
a subset of $\mbox{Im} (H_{\omega_c}^{^\alpha})$, we can find 
a control $\tilde{u}$ in $\mathrm{B}(0,m_\kappa)$ steering 
system \eqref{semili} from $y_0$ to $d_s$ at time $T$ in $\omega_c$.
\end{theorem}

\begin{proof} 
We show that there exists $\rho_\kappa>0$ where,  
for all $d_s\in \mathrm{B}(0,\rho_\kappa)$, the operator 
$\Psi(d_s,\cdot)$ defined from $\mathrm{B}(0,m_\kappa)$ 
into $\mathrm{B}(0,m_\kappa)$ has a fixed point, 
which is then a control steering system \eqref{semili} 
to $d_s$ in $\omega_c$ at time $T$. The proof uses  
the classical fixed point theorem of a contraction mapping.

(i) We show that $\Psi(d_s,\cdot)$ is a contraction mapping.
Let us consider $d_s$ in $\mbox{Im}\:( H^\alpha_{\omega_c})$ 
and $u,v$ in $\mathrm{B}(0,m_\kappa)$, where $m_\kappa$ 
is defined by \eqref{cste1}. Then, 
\begin{equation*}
\begin{split}
||\Psi(d_s,u)-\Psi(d_s,v)||_U
&=||(\chi_{\omega_c}\mathcal{E}_\alpha\ast(\mathcal{F} 
y_u-\mathcal{F} y_v))(T)||_{\mbox{Im}\: (H^\alpha_{\omega_c})}\\
&\leq  ||g_\alpha||_{L^s([0,T])} ||
\mathcal{F} y_u-\mathcal{F} y_v||_{L^p([0,T],X)}\\
&\leq ||g_\alpha||_{L^s([0,T])} 
\underset{(\sigma_i\leq \kappa)}{\sup} F_N(\sigma_1,\sigma_2)
\times ||y_u-y_v||_{L^p([0,T],X^q)}.
\end{split}	
\end{equation*}
It is known from Lemma~\ref{lemma3.6} that $u\longrightarrow y_u$ 
is a Lipschitz mapping from $\mathrm{B}(0,m_\kappa)$ into 
$\mathrm{B}(0,\kappa)$. Then, we obtain that
\begin{equation}
\label{mee}
||\Psi(d_s,u)-\Psi(d_s,v)||_U \leq \dfrac{\mu ||g_\alpha||_{L^s([0,T])}
\underset{(\sigma_i\leq \kappa)}{\sup} F_N(\sigma_1,\sigma_2)}{1-A_1
\underset{(\sigma_i\leq \kappa)}{\sup} F_N(\sigma_1,\sigma_2)} ||u-v||_U.
\end{equation} 
If we consider $\mu ||g_\alpha||_{L^s([0,T])}:=A_2$ in inequality 
\eqref{majoration}, we have 
$$
(A_1+A_2)\underset{(\sigma_i
\leq \kappa)}{\sup} F_N(\sigma_1,\sigma_2)<1. 
$$
Consequently, 
$$
A_s:=\dfrac{\mu ||g_\alpha||_{L^s([0,T])}\underset{(\sigma_i
\leq \kappa)}{\sup} F_N(\sigma_1,\sigma_2)}{1-A_1\underset{(\sigma_i
\leq \kappa)}{\sup} F_N(\sigma_1,\sigma_2)}<1, 
$$
and thus $\Psi(d_s,\cdot)$ is a strict contraction mapping.

(ii) We now show that $\Psi(d_s,\cdot)$ maps 
$\mathrm{B}(0,m_\kappa)$ into  itself.
Consider $u\in\mathrm{B}(0,m_\kappa)$. 
We have $y_u\in \mathrm{B}(0,\kappa)$ and 
$$
\begin{array}{llll}
||\Psi(d_s,u)||_U
&=&||d_s-(\chi_{\omega_c}\mathcal{E}_\alpha\ast 
\mathcal{F} y_u)(T)||_{\mbox{Im}\: (H^\alpha_{\omega_c})}\\
&\leq& ||d_s||_{\mbox{Im}\: (H^\alpha_{\omega_c})}
+||(\chi_{\omega_c} \mathcal{E}_\alpha\ast \mathcal{F} y_u)(T)||_{\mbox{Im}\: ()}\\
& \leq& ||d_s||_{\mbox{Im}\: H^\alpha_{\omega_c}}+||g_\alpha||_{L^s([0,T])} 
\kappa \underset{(\theta\leq \kappa)}{\sup} F_N(\theta,0).
\end{array}
$$
Therefore, if 
$$
||d_s||_{\mbox{Im}\: (H^\alpha_{\omega_c})}
\leq m_\kappa- ||g_\alpha||_{L^s([0,T])} 
\kappa \underset{(\theta\leq \kappa)}{\sup} F_N(\theta,0),
$$
then $\Psi(d_s,u)\in\mathrm{B}(0,m_\kappa)$.
Now, take  
$$
\rho_\kappa=\dfrac{\kappa}{\mu}(1-(A_1+A_2) 
\underset{\theta\leq \kappa}{\sup} \: F_N(\theta,0)), 
$$
which is a positive constant. Then $\varPsi(\mathrm{B}(0,\rho_\kappa),
\mathrm{B}(0,m_\kappa))\subset \mathrm{B}(0,m_\kappa)$.

By the contraction mapping theorem, the existence of a fixed 
point of $\Psi(d_s,\cdot)$ is shown. The theorem is proved.
\end{proof}

By the contraction mapping theorem, 
we can also obtain the following result.

\begin{corollary}
The sequence
\begin{equation}
\label{suite}
\begin{cases}
u_0=0, \\
u_{n+1}= \mathrm{H}^{\alpha^\dag}_{\omega_c}(d_s
-\chi_{\omega_c }(\mathcal{E}_\alpha\ast \mathcal{F} y_{u_n})(T)),      
\end{cases}
\end{equation}
converges  to $\tilde{u}$ in $\mathrm{B}(0,m_\kappa)$.
\end{corollary}


\section{Numerical Approach and Examples}
\label{sec4}

In this section, based on the theoretical results 
of Section~\ref{sec3}, we present an algorithm 
that allows us to find, numerically, a control 
function that steers our system to a target state on 
$\varGamma$ at time $T$: see Algorithm~\ref{euclid}.

\begin{algorithm}
\caption{}\label{euclid}
\begin{algorithmic}[H]
\BState \emph{\bf{Initial Data:}}
\State $\bullet$  $\alpha$,  $T$,$\mathrm{\varGamma}$, 
the desired state on $\mathrm{\varGamma}$ $z_d$, and the 
location of the considered zonal or pointwise actuator;
\State $\bullet$ region $\omega_c$ 
where $\mathrm{\varGamma}\subseteq \partial\omega_c$ and    
$d_s$ the extension of $z_d$  in $\omega_c$;
\State $\bullet$ error estimate $\varepsilon$.

\State Initial datum: $r_1=d_s$.  
		
\BState \emph{\bf{Repeat}}
\State $\bullet$ Solve equation $u_n=H^{\alpha \dag}_{\omega_c} r_n$.
\State 	$\bullet$ Solve the semilinear system controlled 
by $u_n$ and obtain $y_{u_n}(T)$.
\State $\bullet$ Compute $r_{n+1}$ by the formula 
$$
r_{n+1}=r_{n}+(d_s-\chi_{\omega_c} y_{u_{n}}(T)), \quad n\geq 2.
$$
\BState \emph{\bf{Until}}
\State $$||r_{n+1}-r_n||_{Im (\mathrm{H}^\alpha_{\omega_c})}\leq\varepsilon.$$
\end{algorithmic}
\end{algorithm}

We apply our Algorithm~\ref{euclid} 
in two examples of time fractional diffusion systems.

\begin{example} 
\label{ex01}
Let $\Omega=]0,1[\times]0,1[$ and $T=3$.
Consider the two-dimensional fractional system 
with diffusion described as follows:
\begin{equation}
\label{simulation}
\left\{\begin{array}{llll}
^{^{C}}D_{0^+}^{^{0.3}}  \theta(x,y,t) 
-\dfrac{\partial^2}{\partial x^2} \theta(x,y,t)
-\dfrac{\partial^2}{\partial y^2} \theta(x,y,t) 
=\theta^2(x,y,t)+ \chi_{_{D}}u(t)  
&  \text{ in } \Omega\times\left]0,3\right] \\
\dfrac{\partial \theta}{\partial\nu_{A}}(\xi,\nu,t) =0 
&\text{ on } \partial\Omega\times\left]0,3\right]\\
\theta(x,y,0) = 0  &\text{ in } \Omega.
\end{array}  \right.
\end{equation} 
In this example, we consider a bounded control operator as a zonal actuator. 
We take $D=[0,0.2]\times[0.2,0.4]$ as the location of the actuator. 

The desired state on $\varGamma=\{0\}\times[0,0.1]$ 
is taken as $\theta_d(y)=7 y^3-13y^2+3$. Applying Algorithm~\ref{euclid} 
with $\varepsilon=10^{-3}$, $\omega_c=[0,0.3]\times[0,0.1]$,  
and the extension of $z_d$ in $\omega_c$  
$$
d_s(x,y)=10\left(\dfrac{x^3}{46}-\dfrac{x^2}{62}+0.1\right)
\left(7 y^3-13y^2+3\right),
$$
we obtain the results given in Figures~\ref{fig1} and \ref{fig2}.
\begin{figure}[ht!]
\centering
\includegraphics[scale=0.6]{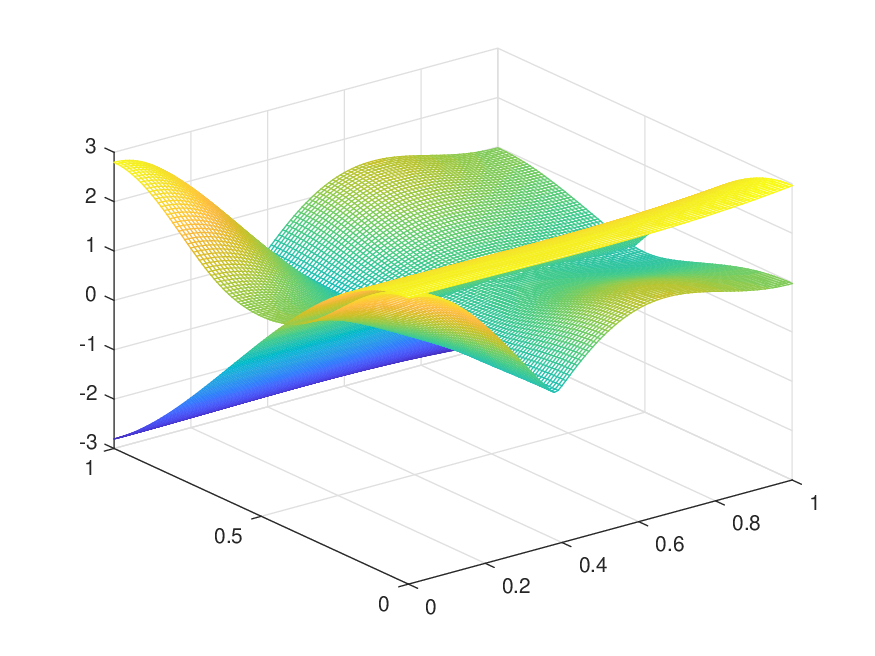}
\caption{Reached state and the state $d_s$ in $\Omega$
for Example~\ref{ex01}.}
\label{fig1}
\end{figure}
\begin{figure}[ht!]
\centering
\includegraphics[scale=0.6]{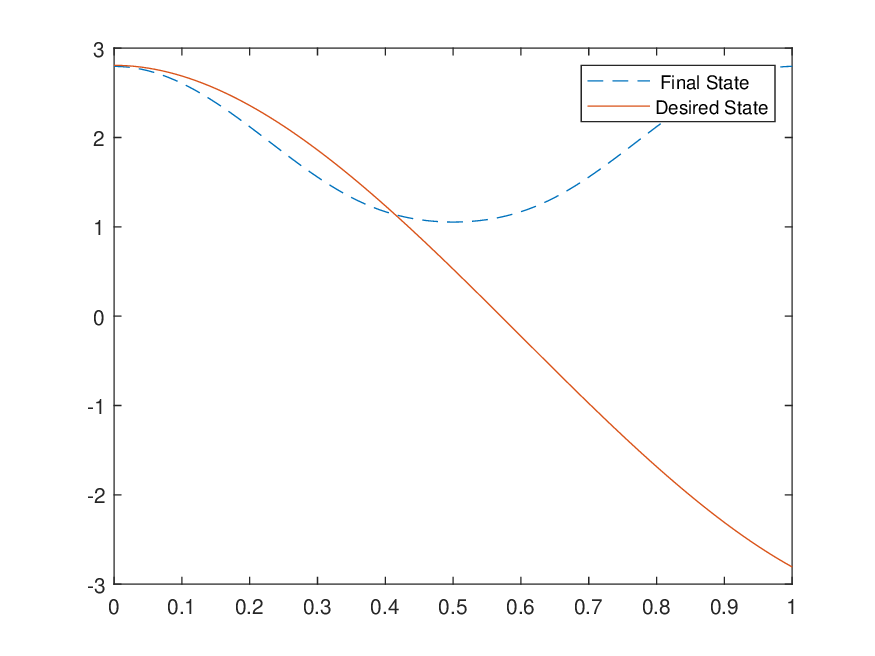}
\caption{Desired state and reached one on $\varGamma$
for Example~\ref{ex01}.}
\label{fig2}
\end{figure}
Figure~\ref{fig1} shows that the state in time $T$ 
is  very close to $d_s$ in the subregion $\omega_c$. 
After the projection on $\mathrm{\varGamma}$, Figure~\ref{fig2} 
shows that the target state is obtained with a reconstruction error 
$1.7\times 10^{-4}$ on $\varGamma$ with cost $||\tilde{u}||_{L^2(0,T)}^2=0.71$.
\end{example}

\begin{example} 
\label{ex02}	
Now we consider the same $\Omega$ as in Example~\ref{ex01}, 
$T=2$, and the following semilinear diffusion system with Caputo
fractional derivative:
\begin{equation}
\label{simulation2}
\begin{cases}
\begin{split}	
^{^{C}}D_{0^+}^{^{0.6}}  z(x,y,t) 
-&\left(\dfrac{\partial^2}{\partial x^2}
+ \dfrac{\partial^2}{\partial y^2}\right) 
z(x,y,t)\\
&=z^2(x,y,t)+\delta_{\{b_1,b_2\}}(x,y)u(t)  
\end{split}
&  \text{ in } \Omega\times\left]0,2\right], \\
\dfrac{\partial z}{\partial\nu_{A}}(\xi,\nu,t) =0 
&\text{ on } \partial\Omega\times\left]0,2\right],\\
z(x,y,0) = 0  &\text{ in } \Omega.
\end{cases}  
\end{equation} 
In this second example, the control operator $B$ is unbounded 
but it satisfies the condition \eqref{condu}. 
In this example, the desired state is 
$z_d(y)=2\left(\dfrac{y^3}{42}-\dfrac{y^2}{1.3}+0.1\right)$. 
Let us consider $\varGamma=\{0\}\times[0,0.3]$, 
$ \omega_c=[0,0.2]\times[0,0.3]$, $b_1=0.48$, $b_2=0.70$, 
$$
d_s(x,y)=20\left(\dfrac{x^3}{4}-\dfrac{x^2}{65}+0.1\right)
\left(\dfrac{y^3}{42}-\dfrac{y^2}{1.3}+0.1\right),
$$  
and $\varepsilon=10^{-3}$. Using Algorithm~\ref{euclid} 
with our data, we obtain the results 
of Figures~\ref{fig3}--\ref{fig5}.
\begin{figure}[ht!]
\centering
\includegraphics[scale=0.6]{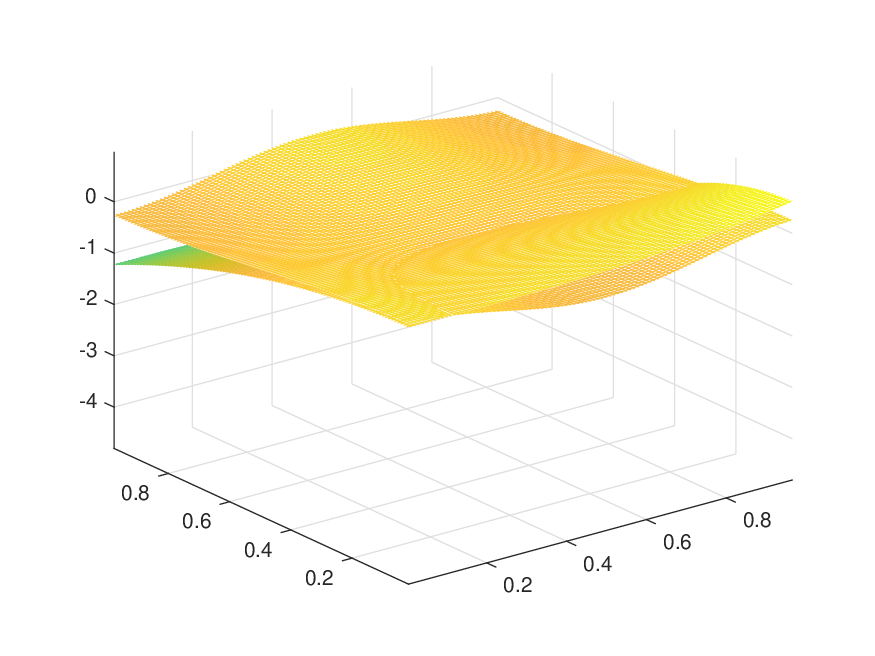}
\caption{Reached state and the state $d_s$ in $\Omega$
for Example~\ref{ex02}.}
\label{fig3}
\end{figure}
\begin{figure}[ht!]
\centering
\includegraphics[scale=0.6]{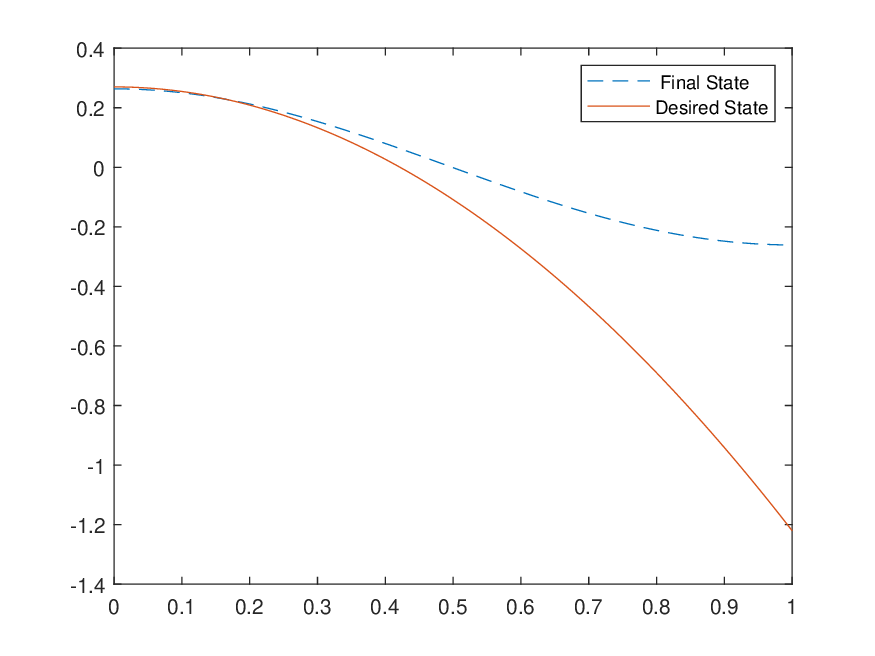}
\caption{Desired state and reached one on $\varGamma$
for Example~\ref{ex02}.}
\label{fig4}
\end{figure}
\begin{figure}[ht!]
\centering
\includegraphics[scale=0.6]{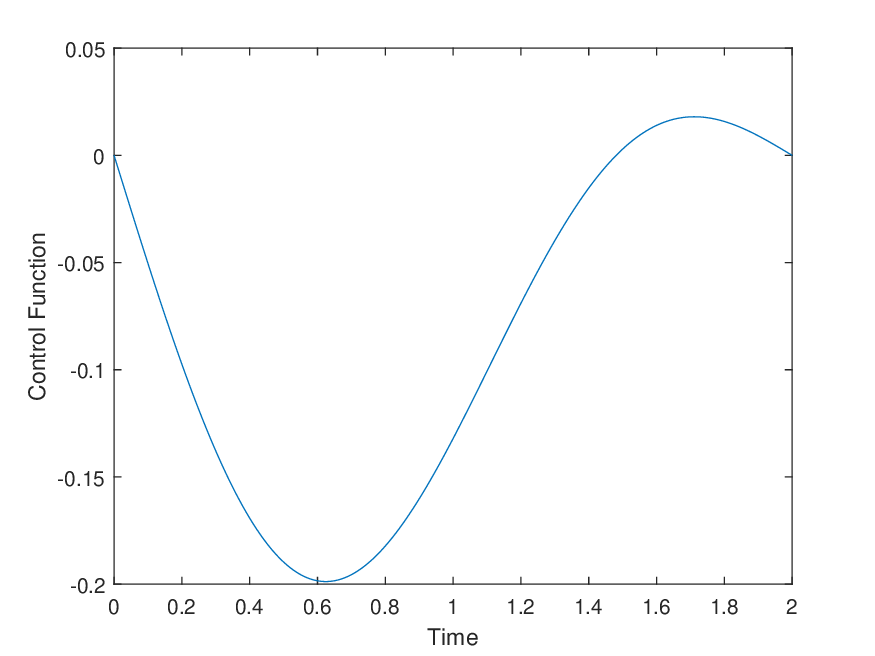}
\caption{Evolution of the control function of Example~\ref{ex02}.}
\label{fig5}
\end{figure}

In Figure~\ref{fig3}, we see that the final state at time $T$ 
is very close to the extension of the target state $d_s$ in $\omega_c$. 
From Figure~\ref{fig4}, we can see that the 
reconstruction error is smaller than  $10^{-3}$.
Figure~\ref{fig5} presents the evolution of the control 
function $\tilde{u}(t)$ with the cost 
$||\tilde{u}||_{L^2(0,T)}^2=2\times 10^{-4}$.
\end{example}


\section{Conclusion}
\label{sec5}

We have proved sufficient conditions for the boundary regional controllability 
of semilinear Caputo fractional systems. More precisely, the  boundary regional 
controllability of a class of semilinear Caputo systems is implied by the 
controllability in a suitable sub-region of the evolution domain. Further, 
the regional controllability problem is transformed into a fixed point problem 
of an appropriate nonlinear operator. Under some conditions in the nonlinear part 
of the system, we have guaranteed the existence of a fixed point of this operator 
and established the regional controllability of the considered system. Finally, 
we have presented two successful numerical simulations to illustrate 
our theoretical study. 

For future work, it would be interesting to investigate the validity 
of the obtained results for other different fractional systems, e.g., 
Hadamard or Caputo--Fabrizio.


\medskip

\noindent\textbf{Acknowledgment.} 
This work was partially supported by CIDMA 
and is funded by the Funda\c{c}\~{a}o para a Ci\^{e}ncia 
e a Tecnologia, I.P. (FCT, Funder ID = 50110000187) 
under Grants UIDB/04106/2020 (\url{https://doi.org/10.54499/UIDB/04106/2020})
and UIDP/04106/2020 (\url{https://doi.org/10.54499/UIDP/04106/2020}). 
Tajani and Torres are also supported through project CoSysM3, 
Reference 2022.03091.PTDC, financially supported 
by national funds (OE) through FCT/MCTES.




\begin{thebibliography}{xx}
	
\bibitem{MR3904404}
M. S. Ali, M. Shamsi, H. Khosravian-Arab, D. F. M. Torres and F. Bozorgnia, 
\newblock A space-time pseudospectral discretization method for solving diffusion 
optimal control problems with two-sided fractional derivatives.
\textit{J. Vib. Control} \textbf{25}(5) (2019), 1080--1095.
{\tt arXiv:1810.05876}

\bibitem{Vector-valued}
W. Arendt, C. J. K. Batty, M.  Hieber and F.  Neubrander, 
\textit{Vector-valued Laplace transforms and Cauchy problems.}
Second Edition, Birkh\"{a}user, Basel, 2011.

\bibitem{MR3337805}
A. Bacciotti, 
Periodic asymptotic controllability of switched systems,
\textit{Lib. Math. (N.S.)} \textbf{34}(1) (2014), 23--46.

\bibitem{MR3684684} 
A. Bacciotti, 
\newblock Bounded-input-bounded-state stabilization 
of switched processes and periodic asymptotic controllability.
\textit{Kybernetika} \textbf{53}(3) (2017), 530--544.

\bibitem{MR2784697}
A. Bacciotti and L. Mazzi, 
\newblock Asymptotic controllability by means of eventually periodic switching rules,
\textit{SIAM J. Control Optim.} \textbf{49}(2) (2011), 476--497.

\bibitem{anamalous}
W. Chen,  H. G. Sun, X. Zhang  and  D. Korosak,
\newblock Anomalous diffusion modeling by fractal and fractional derivatives.
\textit{Comput. Math. Appl.} \textbf{59}(5) (2010), 1754--1758.

\bibitem{cinter3}
V. Do,
\newblock A note on approximate controllability of semilinear systems. 
\textit{Syst. Control. Lett.} \textbf{12} (1989), 365--371.

\bibitem{borai}
M. M. El-Borai, 
\newblock Some probability densities 
and fundamental solutions of fractional evolution equations.
\textit{Chaos Solitons Fractals} \textbf{14} (2002), 433--440.

\bibitem{regana} 
G. Fudong, C. YangQuan and K. Chunhai, 
\textit{Regional analysis of time-fractional diffusion processes}, 
\newblock Springer International Publishing, 2018.

\bibitem{grcont}
F. Ge, Y. Q. Chen and C. Kou,
\newblock Regional gradient controllability of sub-diffusion processes.
\textit{J. Math. Anal. Appl.} \textbf{440}(2) (2016), 865--884.

\bibitem{regbund}
F. Ge, Y. Q. Chen and C. Kou,
\newblock Regional boundary controllability of time fractional diffusion processes. 
\textit{IMA J. Math. Contr. Infor.} \textbf{34}(3) (2017), 871--888.

\bibitem{cinter2}
R. Glowinski and J. L. Lions,
\newblock Exact and approximate controllability for distributed parameter systems. 
\textit{Acta  Numer.} \textbf{4} (1995), 159--328.

\bibitem{theor}
E. Hernandez,  D. O'Regan and K. Balachandran, 
\newblock On recent developments in the theory of abstract 
differential equations with fractional derivatives.
\textit{Nonlinear Anal. } \textbf{73}(10) (2010), 3462--3471.

\bibitem{hilfer}
R. Hilfer, 
\textit{Applications of fractional calculus in physics.} 
\newblock Default Book Series, 2000.

\bibitem{nonl2}
R. Joice Nirmala and K. Balachandran,
\newblock Controllability of fractional nonlinear systems in Banach spaces.
\textit{J. Appl. Nonlinear Dyn.} \textbf{5}(4) (2016), 485--494.

\bibitem{kalman}
R. E. Kalman,
\newblock Controllablity of linear dynamical systems.
\textit{Contrib. Diff. Equ.} \textbf{1} (1963), 190--213.

\bibitem{bout1}
T. Karite and A. Boutoulout, 
\newblock Regional constrained controllability for parabolic semilinear systems.
\textit{Int. J. Pure Appl. Math.} \textbf{113}(1) (2017), 113--129.

\bibitem{bout2}
T. Karite and A. Boutoulout,
\newblock Regional boundary controllability of semi-linear parabolic 
systems with state constraints.
\textit{Int. J. Dyn. Syst. Differ. Equ.} \textbf{8}(1-2) (2018), 150--159.

\bibitem{11} 
A. A. Kilbas, H. M.Srivastava and J. J. Trujillo, 
\textit{Theory and applications of fractional differential equations}. 
\newblock Elsevier Science B.V., Amsterdam, 2006. 

\bibitem{cont2}
J. Klamka,
\newblock Stochastic controllability of linear systems with delay in control.
\textit{Bull. Pol. Acad. Sci.} \textbf{55} (2007), 23--29.

\bibitem{cinte}
J. L. Lions,
\newblock On the controllability of distributed systems.
\textit{Proc. Natl. Acad. Sci.} \textbf{94} (1997), 4828--4835.

\bibitem{nonlinear}
Z. H. Liu and J. H. Sun,
\newblock Nonlinear boundary value problems of fractional functional 
integrodifferential equations.
\textit{Comput. Math. Appl.} \textbf{64} (2012), 3228--3234.

\bibitem{diff}
Y. Luchko,
\newblock Fractional calculus models for the anomalous diffusion 
processes and their analysis, Application of Mathematics in Technical and Natural Sciences.
\textit{AIP Conf. Proc.} \textbf{1301}(1) (2010), 623--635.

\bibitem{cont1}
N. I. Mahmudov,
\newblock Approximate controllability of semilinear deterministic 
and stochastic evolution equations in abstract spaces.
\textit{SIAM J.	Control Optim.} \textbf{42}(5) (2003), 1604--1622.

\bibitem{major}  
N. I. Mahmudov and S. Zorlu, 
\newblock On the approximate controllability of fractional 
evolution equations with compact analytic semigroup.
\textit{J. Comput. Appl. Math.} \textbf{259} (2014) 194--204.

\bibitem{miller}  
K. S. Miller and B. Ross, 
\textit{An introduction to the fractional calculus
and fractional differential equations.}
John Wiley and Sons, New York, 1993.

\bibitem{app}
K. B. Oldham and J. Spanier, 
\textit{Fractional Calculus: Theory and Applications of Differentiation 
and Integration to Arbitrary Order}. 
\newblock Academic Press, London, 1974.

\bibitem{pazy}
A. Pazy, 
\textit{Semigroups of linear operators and applications 
to partial differential equations.} 
\newblock Springer Science, Berlin, 2012.

\bibitem{sbati}
J. Sabatier,  O. P. Agarwal and J. A. T. Machado,
\textit{Advances in fractional calculus: theoretical developments 
and applications in physics and engineering.}
\newblock Springer, Dordrecht, Netherlands, 2007.

\bibitem{nonl1}
R. Sakthivel, Y. Ren and N. I. Mahmudov,
\newblock On the approximate controllability of semilinear fractional differential systems. 
\textit{Comput. Math. Appl.} \textbf{62} (2011), 1451--1459.

\bibitem{solution}
X.-B. Shu and Y. J. Shi, 
\newblock A study on the mild solution of impulsive fractional evolution equations.
\textit{Appl. Math. Comput.} \textbf{273} (2016), 465--476.

\bibitem{MyID:441}
M. R. Sidi Ammi, I. Jamiai and D. F. M. Torres,
\newblock A finite element approximation for 
a class of Caputo time-fractional diffusion equations.
\textit{Comput. Math. Appl.} \textbf{78}(5) (2019), 1334--1344.
{\tt arXiv:1905.10657}

\bibitem{mea}
A. Tajani,  F.-Z. El Alaoui and A. Boutoulout, 
\newblock Regional Controllability of Riemann-Liouville 
Time-Fractional Semilinear Evolution Equations.
\textit{Math. Probl. Eng.}  \textbf{2020} (2020), Art. ID 5704251, 7~pp.

\bibitem{me}
A. Tajani, F.-Z. El Alaoui and A. Boutoulout, 
\newblock Regional Controllability of a Class of Time-Fractional Systems. 
In: Hammouch Z., Dutta H., Melliani S., Ruzhansky M. (eds) 
\textit{Nonlinear Analysis: Problems, Applications and Computational Methods}. 
Lecture Notes in Networks and Systems, Springer, \textbf{168} (2021), 141--155.

\bibitem{cont}
J. R. Wang and Y. Zhou,
\newblock A class of fractional evolution equations and optimal controls.
\textit{Nonlinear Anal. Real World Appl.} \textbf{12} (2011), 262--272.

\bibitem{majorit}
J. Wang and Y. Zhou,
\newblock Analysis of nonlinear fractional control systems in Banach spaces.
\textit{Nonlinear Anal.} \textbf{74} (2011), 5929--5942.

\bibitem{MyID:386}
X.-J. Yang, H. M. Srivastava, D. F. M. Torres and A. Debbouche,
\newblock General fractional-order anomalous diffusion with nonsingular power-law kernel.
\textit{Thermal Science} \textbf{21} (2017), Suppl.~1, S1--S9.

\bibitem{55} 
E. Zerrik, A. El Jai and A. Boutoulout, 
\newblock Actuators and regional boundary controllability of parabolic system.  
\textit{Int. J. Sys. Sci.} \textbf{31} (2000), 73--82.

\bibitem{kamal}
E. Zerrik and A. Kamal,
\newblock Output controllability for semi-Linear distributed systems.
\textit{J. Dyn. Control Syst.} \textbf{13} (2007), 289--306.

\bibitem{zhoucon} 
H. X. Zhou,
\newblock Approximate controllability for a class of semilinear abstract equations.
\textit{SIAM J.	Control Optim.} \textbf{22} (1983), 405--422.

\bibitem{zhouex}
Y. Zhou and F. Jiao,
\newblock Existence of mild solutions for fractional neutral evolution equations.
\textit{Comput. Math. Appl.} \textbf{59} (2010), 1063--1077.

\end{thebibliography}
\end{document}